\documentclass[12pt]{amsart}

\usepackage[usenames]{color}
\usepackage{fullpage,url,mathrsfs,stmaryrd}

\usepackage{relsize}
\usepackage[bbgreekl]{mathbbol}
\usepackage{amsfonts}

\DeclareSymbolFontAlphabet{\mathbb}{AMSb}
\DeclareSymbolFontAlphabet{\mathbbl}{bbold}
\newcommand{\prism}{{\mathlarger{\mathbbl{\Delta}}}}

\newcommand{\BG}{{\mathbb{G}}}

\usepackage{amssymb}
\usepackage[all]{xy}
\usepackage{mathrsfs}
\usepackage{enumerate}
\usepackage{amscd}

\usepackage{bbm}

\DeclareMathOperator{\fin}{{fin}}

\DeclareMathOperator{\qcqs}{{qcqs}}

\DeclareMathOperator{\WCart}{{WCart}}

\DeclareMathOperator{\Frac}{{Frac}}
\DeclareMathOperator{\FrmSch}{{FrmSch}}
\newcommand{\FSch}{{\FrmSch^\flat}}
\newcommand{\qFSch}{{\FrmSch^\flat_{\qcqs}}}

\DeclareMathOperator{\Isoc}{{Isoc}}

\DeclareMathOperator{\perf}{{perf}}
\DeclareMathOperator{\Tot}{{Tot}}

\newcommand{\HHom}{\underline{\on{Hom}}}

\DeclareMathOperator{\cris}{{cris}}

\DeclareMathOperator{\Cris}{{Cris}}

\newcommand{\QCoh}{{\on{QCoh}}}

\newcommand{\cA}{{\mathcal A}}
\newcommand{\cB}{{\mathcal B}}

\newcommand{\cF}{{\mathcal F}}

\newcommand{\cO}{{\mathcal O}}
\newcommand{\cP}{{\mathcal P}}

\newcommand{\sB}{{\mathscr B}}

\newcommand{\sF}{{\mathscr F}}
\newcommand{\sG}{{\mathscr G}}

\newcommand{\sR}{{\mathscr R}}

\newcommand{\sX}{{\mathscr X}}
\newcommand{\sY}{{\mathscr Y}}
\newcommand{\sZ}{{\mathscr Z}}

\newcommand{\nc}{\newcommand}

\nc\wh{\widehat}

\nc\on{\operatorname}

\nc\Gr{\on{Gr}}

\nc\Fl{\on{Fl}}

\newtheorem{cor}[subsubsection]{Corollary}
\newtheorem{lem}[subsubsection]{Lemma}
\newtheorem{prop}[subsubsection]{Proposition}

\newtheorem{thm}[subsubsection]{Theorem}

\theoremstyle{remark}
\newtheorem{rem}[subsubsection]{Remark}

\newcommand{\BF}{{\mathbb{F}}}
\newcommand{\BN}{{\mathbb{N}}}
\newcommand{\BQ}{{\mathbb{Q}}}

\newcommand{\BZ}{{\mathbb{Z}}}

\DeclareMathOperator{\Bun}{{Bun}}

 \DeclareMathOperator{\Spf}{{Spf}}

\newcommand{\limto}{{\displaystyle\lim_{\longrightarrow}}}
\newcommand{\rightlim}{\mathop{\limto}}

\newcommand{\leftlim}{\mathop{\displaystyle\lim_{\longleftarrow}}}
\newcommand{\limfromn}{\leftlim\limits_{\raise3pt\hbox{$n$}}}
\newcommand{\limton}{\rightlim\limits_{\raise3pt\hbox{$n$}}}

\newcommand{\rightlimit}[1]{\mathop{\lim\limits_{\longrightarrow}}\limits%
                    _{\raise3pt\hbox{$\scriptstyle #1$}}}

\newcommand{\leftlimit}[1]{\mathop{\lim\limits_{\longleftarrow}}\limits%
                    _{\raise3pt\hbox{$\scriptstyle #1$}}}

\newcommand{\epi}{\twoheadrightarrow}
\newcommand{\iso}{\buildrel{\sim}\over{\longrightarrow}}
\newcommand{\mono}{\hookrightarrow}

\DeclareMathOperator{\Coker}{{Coker}}

\DeclareMathOperator{\End}{{End}}

\DeclareMathOperator{\Hom}{{Hom}}

\DeclareMathOperator{\Ker}{{Ker}} \DeclareMathOperator{\id}{{id}}

\DeclareMathOperator{\Spec}{{Spec}}

\theoremstyle{definition}

\newtheorem{defin}[subsubsection]{Definition}

\newtheorem{ex}[subsubsection]{Example}

\numberwithin{equation}{section}

\newcommand{\Fr}{\operatorname{Fr}}

\newcommand{\Acris}{{\cA}}
\newcommand{\Font}{{\cF}}
\newcommand{\Pow}{{\cP}}

\begin{document}
\title[A stacky approach to crystals]{A stacky approach to crystals}
\author{Vladimir Drinfeld}
\address{University of Chicago, Department of Mathematics, Chicago, IL 60637, USA}
\email{drinfeld@uchicago.edu}
\dedicatory{To Professor C.~N.~Yang with deepest admiration}

\begin{abstract}
Inspired by a theorem of Bhatt-Morrow-Scholze, we develop a stacky approach to crystals and isocrystals on ``Frobenius-smooth" schemes over $\BF_p\,$. This class of schemes goes back to Berthelot-Messing and contains all smooth schemes over perfect fields of characteristic $p$. 

To treat isocrystals, we prove some descent theorems for sheaves of Banachian modules, which could be interesting in their own right.
\end{abstract}

\keywords{$F$-isocrystal, slope filtration, Tannakian category, formal group, semiperfect ring}
\subjclass[2010]{Primary 14F30}

\maketitle

\section{Introduction}   \label{s:2Introduction}

Fix a prime $p$. 

\subsection{A theorem of Bhatt-Morrow-Scholze}
Let $X$ be a smooth scheme over a perfect field $k$ of characteristic $p$. Let $X_{\perf}$ be the perfection of $X$, i.e.,
\[
X_{\perf}:=\underset{\longleftarrow}\lim (\ldots\overset{\Fr}\longrightarrow X\overset{\Fr}\longrightarrow X\overset{\Fr}\longrightarrow X).
\]
Let $W(X_{\perf})$ be the $p$-adic formal scheme whose underlying topological space is that of $X_{\perf}$ (or of $X$) and whose structure sheaf is obtained by applying to 
$\cO_{X_{\perf}}$ the functor of $p$-typical Witt vectors. Theorem~1.10 of \cite{BMS} can be reformulated as a canonical isomorphism
\[
R\Gamma_{\cris} (X,\cO )=R\Gamma (W(X_{\perf})/\sG ,\cO ),
\]
where $\sG$ is a certain flat affine groupoid (in the category of $p$-adic formal schemes) acting on $W(X_{\perf})$, and $W(X_{\perf})/\sG$ is the quotient stack\footnote{The stacks $(W({X_{\perf})}/\sG)\otimes_{\BZ_p}(\BZ/p^r\BZ)$ are usually \emph{not} Artin stacks 
because the two canonical morphisms $\sG\to W({X_{\perf})}$ usually have infinite type.}.
If $X$ is affine this gives an explicit complex computing $R\Gamma_{\cris} (X,\cO )$ (the complex is constructed using the nerve of $\sG$).

One can define $\sG$ to be the unique groupoid acting on $W(X_{\perf})$ such that the morphism $\sG\to W(X_{\perf})\times W(X_{\perf})$ is obtained by taking the divided power envelope of  
the ideal of the closed subscheme 
$$X_{\perf}\times_X X_{\perf}\subset W(X_{\perf})\times W(X_{\perf}).$$
We prefer to define the groupoid $\sG$ by describing its nerve $\Acris_{\bullet}$ using Fontaine's func\-tor~$A_{\cris}$ (see \S\ref{sss:Acris}, in which we follow \cite{BMS}).

\subsection{A generalization}
Our first main result is Theorem~\ref{t:crystals via stacks}. It says that if $X$ is as above then a crystal of quasi-coherent $\cO$-modules on the absolute crystalline site of $X$ is the same as an $\cO$-module $M$ on the stack $W(X_{\perf})/\sG$, and the complex $R\Gamma (W(X_{\perf})/\sG ,M)$ identifies with the cohomology of the corresponding crystal. 

We think of crystalline cohomology not in terms of the de Rham complex but in terms of the bigger and more tautological \v {C}ech-Alexander complex. This approach makes Theorem~\ref{t:crystals via stacks} almost obvious.

Following Berthelot-Messing \cite[\S 1]{BM}, we allow $X$ to be any $\BF_p$-scheme $X$ which is \emph{Frobenius-smooth} in the sense of \S\ref{ss:morally smooth}. E.g., if $k$ is a perfect $\BF_p$-algebra then $\Spec k[[x_1,\ldots ,x_n]]$ or any smooth scheme over $\Spec k$ is allowed.

\begin{rem}   \label{r:prismatization}
As explained in Appendix~\ref{appendix}, for any Frobenius-smooth $\BF_p$-scheme $X$, the stack $W(X_{\perf})/\sG$ canonically identifies with the \emph{prismatization} of $X$. The notion of prismatization is not used in the main body of this article (in fact, it did not exist when the original version of the article was written).
\end{rem}

\subsection{Isocrystals}
Now let $X$ be a Noetherian $\BF_p$-scheme. For such schemes Frobenius-smoothness is equivalent to $\Omega^1_X$ being locally free and coherent (see \S\ref{ss:noetherian morally smooth}); let us assume this. E.g., if $k$ is a perfect field of characteristic $p$ then $\Spec k[[x_1,\ldots ,x_n]]$ or any smooth scheme over $\Spec k$ is allowed. 

\subsubsection{What we mean by an isocrystal}
Consider the category of crystals of finitely generated quasi-coherent $\cO$-modules on the absolute crystalline site of $X$. Tensoring it by $\BQ$, one gets a category denoted by $\Isoc (X)$, whose objects are called \emph{isocrystals}. (Thus our isocrystals are not necessarily convergent in the sense of \cite{O}.)

\subsubsection{The result on isocrystals}
Our second main result (Theorem~\ref{t:isocrystals via stacks}) provides a canonical equivalence
\begin{equation}  
\Isoc (X)\iso\Bun_{\BQ}(W(X_{\perf})/\sG ),
\end{equation}
where the right-hand side is, so to say, the category of vector bundles on $(W(X_{\perf})/\sG )\otimes_{\BZ_p}\BQ_p$ (more details are explained in \S\ref{sss:2Bun_Q} below).

\subsubsection{``Banachian games" and $\Bun_\BQ\,$}   \label{sss:2Bun_Q}
For any $\BZ_p$-flat $p$-adic formal scheme\footnote{See \S\ref{sss:formal stacks} for the precise meaning of these words. Note that $\sY$ does not have to be Noetherian.} $\sY$, let $\Bun_{\BQ}(\sY )$ denote the category of finitely generated locally projective $(\cO_{\sY}\otimes\BQ )$-modules. We prove that the assignment $\sY\mapsto\Bun_{\BQ}(\sY )$ satisfies fpqc descent (see Proposition~\ref{p:descent for Bun_Q}). Because of this, the meaning of $\Bun_{\BQ}(W(X_{\perf})/\sG)$ is clear.

More generally, in \S\ref{ss:Banachian games} we prove fpqc descent for the category of sheaves of Banachian\footnote{By a Banachian space over $\BQ_p$ we mean a complete topological vector space such that the topology comes from a non-Archimedean norm.} $(\cO_{\sY}\otimes\BQ)$-modules, denoted by $\QCoh^\flat (\sY )\otimes\BQ$ (see Corollary~\ref{c:descent for QCohQ}).

\subsection{Acknowledgements}
I thank A.~Beilinson, P.~Berthelot, B.~Bhatt, M.~D'Addezio, A.~Mathew, A.~Ogus, N.~Rozenblyum, and M.~Temkin for valuable advice and references.
The author's research was partially supported by NSF grants DMS-1303100 and DMS-2001425.

\section{Crystals and crystalline cohomology} \label{s:Crystals and crystalline}
\subsection{A class of schemes}  \label{ss:morally smooth}

\begin{lem} \label{l:morally smooth}
The following properties of an $\BF_p$-scheme $X$  are equivalent:

(i) the Frobenius morphism $\Fr:X\to X$ is syntomic (i.e., flat, of finite presentation, and with each fiber being a locally complete intersection);

(ii) every point of $X$ has a neighborhood isomorphic to $\Spec B$ for some $\BF_p$-algebra $B$ with a finite $p$-basis in the sense of \cite[Def.~1.1.1]{BM} (i.e., there exist $x_1,\ldots , x_d\in B$ such that every element of $B$ can be written uniquely as $\sum\limits_{\alpha\in J}b_\alpha^px^\alpha$, $b_\alpha\in B$, where $J:=\{0,1,\ldots ,p-1\}^d$ and for $\alpha=(\alpha_1,\ldots ,\alpha_d)\in J$ one sets $x^\alpha:=x_1^{\alpha_1}\cdot\ldots\cdot x_d^{\alpha_d}$). 
\end{lem}

\begin{proof} 
Let us only explain why (i) implies that the cotangent complex of $X$ over $\BF_p$ is a vector bundle. The relative cotangent complex for $\Fr:X\to X$ is a perfect complex in degrees 0, -1. On the other hand, it is the cone of the map $\Fr^*L_X\to L_X$, where $L_X$ is the cotangent complex of $X$ over $\BF_p$. But this map is zero. So we see that $L_X\oplus\Fr^*L_X[1]$ is a perfect complex in degrees 0, -1. So $L_X$ is a vector bundle.
\end{proof} 

$\BF_p$-schemes satisfying the equivalent conditions of Lemma~\ref{l:morally smooth} will be called \emph{Frobenius-smooth}.
We are mostly interested in the following two classes of such schemes.

\begin{ex}   \label{ex:1morally smooth}
$X$ is a smooth scheme over a perfect field $k$ of characteristic $p$.
\end{ex}

\begin{ex}   \label{ex:2morally smooth}
$X=\Spec k[[x_1,\ldots ,x_n]]$, where $k$ is a perfect field of characteristic $p$.
\end{ex}

\begin{rem}
The equivalent conditions of Lemma~\ref{l:morally smooth} still hold if in Examples \ref{ex:1morally smooth}-\ref{ex:2morally smooth} one replaces ``perfect field" with
``perfect $\BF_p$-algebra".
\end{rem}

\begin{rem}
In \cite[\S 1]{BM} Berthelot and Messing study $\BF_p$-algebras with a \emph{not necessarily finite} $p$-basis; they prove that in some respects they are as good as smooth $\BF_p$-algebras. More results of this type can be found in \cite[\S 1]{dJ}. In particular, any $\BF_p$-algebra with a $p$-basis is formally smooth (see \cite[\S 1.1.1]{BM}), and moreover, its cotangent complex is a free module concentrated in degree $0$ (see \cite[Lemma 1.1.2]{dJ}).
\end{rem}

\begin{lem} \label{l:F-finiteness}
Let $X$ be an $\BF_p$-scheme   whose Frobenius endomorphism is finite. Let $n\in\BN$. Let $\Delta\subset X^n$ be the diagonal $X$ and $\Delta':=(\Fr_{X^n})^{-1}(\Delta)\subset X^n$. Let $I\subset\cO_{\Delta'}$ be the ideal of $\Delta\subset\Delta'$. Then $I$ is finitely generated.
\end{lem}

Let us recall that according to the definition from EGA, ``finitely generated" really means ``locally finitely generated".

\begin{proof} 
Let $\pi:\Delta'\mono X^n\to X$ be one of the $n$ projections. The morphism $\pi$ is finite because $\Fr_X$ is.
So the middle term of the exacts sequence
\[
0\to \pi_*I\to\pi_*\cO_{\Delta'}\to\cO_X\to 0
\]
is a finitely generated $\cO_X$-module. Since our exact sequence is locally split, this implies that
$\pi_*I$ is a  finitely generated $\cO_X$-module. So $I$ is a finitely generated ideal in $\cO_{\Delta'}$.
\end{proof}

\begin{cor} \label{c:F-finiteness}
In the situation of Lemma~\ref{l:F-finiteness} the ideal  $I$ is nilpotent on every quasi-compact open subset of $X$. \qed
\end{cor}

\subsection{Some simplicial formal schemes}
Let $X$ be an $\BF_p$-scheme.

\subsubsection{The simplicial scheme $\Pow_{\bullet}\,$}  \label{sss:Pow}
Let $X_{\perf}$ be the perfection of $X$, i.e.,
\[
X_{\perf}:=\underset{\longleftarrow}\lim (\ldots\overset{\Fr}\longrightarrow X\overset{\Fr}\longrightarrow X\overset{\Fr}\longrightarrow X).
\]

For an integer $n\ge 0$,  let $[n]:=\{0,1,\ldots, n\}$ and let $\Pow_n\subset X_{\perf}^{[n]}$ be the preimage of the diagonal $X\mono X^{[n]}$; equivalently, $\Pow_n$ is the fiber product (over $X$) of $n+1$ copies of $X_{\perf}$. The schemes $\Pow_n$ form a simplicial scheme $\Pow_{\bullet}=\Pow_{\bullet}(X)$ over $X$. 

The scheme $\Pow_0=X_{\perf}$ is perfect. For any $n$, the scheme $\Pow_n$ is \emph{semiperfect} in the sense of \cite{SW}, i.e., the Frobenius endomorphism of $\Pow_n$ is a closed embedding. Note that the underlying topological space of $\Pow_n$ is that of $X$. As usual, the structure sheaf of $\Pow_n$ is denoted by $\cO_{\Pow_n}$.

\subsubsection{The simplicial formal scheme $\Font_{\bullet}\,$}   \label{sss:Font}
Let $\Font_n$ be the ``Fontainization" of $\Pow_n$, i.e., 
\[
\Font_n=\underset{\longrightarrow}\lim ( \Pow_n\overset{\Fr}\longrightarrow  \Pow_n\overset{\Fr}\longrightarrow \Pow_n\overset{\Fr}\longrightarrow \ldots).
\]
Each $\Font_n$ is a perfect formal scheme over $\BF_p\,$, whose underlying topological space is that of~$X$. 
These formal schemes form a simplicial formal scheme $\Font_{\bullet}=\Font_{\bullet}(X)$.

Corollary~\ref{c:F-finiteness} implies that if $\Fr_X$ is finite then $\Font_n$ is equal to the formal completion of $X_{\perf}^{[n]}=X_{\perf}^{n+1}$
along the closed subscheme $\Pow_n\subset X_{\perf}^{[n]}\,$.

\subsubsection{The simplicial formal scheme $\Acris_{\bullet}\,$}   \label{sss:Acris}
Let $\Acris_n$ be the $p$-adic formal scheme obtained from $\Pow_n$ by applying Fontaine's functor $A_{\cris}$ 
(see \cite[\S 2.2]{F94} or \cite[\S 2.4]{Dr}). 
So the underlying topological space of $\Acris_n$ is that of $X$, and the structure sheaf $\cO_{\Acris_n}$ is the $p$-adic completion of the PD hull\footnote{The definition of PD hull is given in \cite[\S I.2.3]{B74}. Throughout this article, PD hulls are taken in the category of PD algebras over $(\BZ_p,p\BZ_p)$ (rather than over $(\BZ_p,0)$). In other words, we assume that $\gamma_n(p)$ is equal to $p^n/n!\in p\BZ_p$.} of the surjection $W(\cO_{\Font_n})\epi\cO_{\Pow_n}$ (as usual, PD stands for ``divided powers" and $W$ for the $p$-typical Witt vectors). The formal schemes $\Acris_n$ form a simplicial formal scheme $\Acris_{\bullet}=\Acris_{\bullet}(X)$. 

By functoriality, the morphism $\Fr_X :X\to X$ induces a canonical simplicial morphism $F:\Acris_{\bullet}(X)\to \Acris_{\bullet}(X)$. The morphism $F:\Acris_n\to \Acris_n$ is usually \emph{not an isomorphism} if $n>0$ (if $n=0$ it is an isomorphism because  $\Acris_0=W(X_{\perf})$).

\subsection{Notation and terminology related to quasi-coherent sheaves}  \label{ss:qcoh notation}

\subsubsection{$p$-adic formal schemes and stacks}  \label{sss:formal stacks}
By a $p$-adic formal scheme (resp.~$p$-adic formal stack) $\sY$ we mean a sequence of schemes (resp.~stacks) $\sY_n$ equipped with isomorphisms $\sY_{n+1}\otimes_{\BZ}\BZ/p^n\BZ\iso \sY_n$.
(Then $\sY_n$ is over $\BZ/p^n\BZ$ and $\sY_{n+1}\otimes_{\BZ}\BZ/p^n\BZ=\sY_{n+1}\otimes_{\BZ/p^{n+1}\BZ}\BZ/p^n\BZ$). In this situation we often write $\sY\otimes \BZ/p^n\BZ$ instead of $\sY_n$.

$\sY$ is said to be $\BZ_p$-flat if each $\sY_n$ is flat over $\BZ/p^n\BZ$.

\subsubsection{The notation $\QCoh (\sY)$}   \label{sss:qcoh notation}
If $\sY$ is a scheme or stack we write $\QCoh (\sY)$ for the category of quasi-coherent sheaves on $\sY$.

Now suppose that $\sY$ and $\sY_n$ are as in \S\ref{sss:formal stacks}. Then we write $\QCoh (\sY)$ for the projective limit of the categories $\QCoh (\sY_n)$ with respect to the pullback functors $\QCoh (\sY_{n+1})\to\QCoh (\sY_n)$ (so an object of $\QCoh (\sY)$ is a sequence of objects $\cF_n\in\QCoh (\sY_n)$ with isomorphisms $\cF_{n+1}/p^n\cF_{n+1}\iso\cF_n$). Note that $\QCoh (\sY)$ is a Karoubian additive category (but usually not an abelian one). We often write $\cF/p^n\cF$ instead of $\cF_n$.

\subsubsection{$\BZ_p$-flatness}   \label{sss:qcoh flatness}
Let $\sY$ and $\sY_n$ be as in \S\ref{sss:formal stacks} and let $\cF\in\QCoh (\sY)$. We say that $\cF$ is \emph{$\BZ_p$-flat} if each $\cF_n$ is flat over $\BZ/p^n\BZ$. 
Let $\QCoh^\flat  (\sY)\subset\QCoh (\sY )$ be the full subcategory of $\BZ_p$-flat objects.

\subsubsection{Finite generation}   \label{sss:qcoh-fin}
We say that $\cF\in\QCoh (\sY)$ is \emph{finitely generated} if $\cF_n$ is finitely generated\footnote{Let us recall that according to the definition from EGA, ``finitely generated" really means ``locally finitely generated".} for each $n$ (or equivalently, for $n=1$). Let $\QCoh_{\fin} (\sY )\subset\QCoh (\sY )$ be the full subcategory of finitely generated objects.
Let $\QCoh_{\fin}^\flat  (\sY ):=\QCoh_{\fin} (\sY )\cap\QCoh^\flat  (\sY)$.

\subsubsection{Cohomology}   \label{sss:qcoh cohomology}
Let $\sY$ and $\sY_n$ be as in \S\ref{sss:formal stacks}. Let $\cF$ and $\cF_n$ be as in \S\ref{sss:qcoh notation}. Then we define $R\Gamma (\sY,\cF )$ to be the homotopy projective limit of $R\Gamma (\sY_n,\cF_n )$.

\subsubsection{Equivariant objects}   \label{sss:equivariant objects}
Let $\Gamma\rightrightarrows \sY $ be a groupoid in the category of $p$-adic formal schemes. Assume that the canonical morphisms $\Gamma\to \sY$ are flat, quasi-compact, and quasi-separated. Then for each $n$ we have the quotent stack $\sY_n/\Gamma_n$ (the words ``quotient" and ``stack" are understood in the sense of the fpqc topology). The stacks $\sY_n/\Gamma_n$ form a $p$-adic formal stack $\sY/\Gamma$. In this situation objects of $\QCoh (\sY/\Gamma )$ are also called \emph{$\Gamma$-equivariant objects} of $\QCoh (\sY)$.

\subsubsection{Objects of $\QCoh (\sY )$ as sheaves}  \label{sss:bad&good news}
Let $\sY$ be a $p$-adic formal scheme and $\cF\in\QCoh (\sY)$. Let $\cF_\infty$ be the projective limit of the sheaves $\cF_n=\cF/p^n\cF$ in the category of presheaves on the topological space $\sY$; it is clear that $\cF_\infty$ is a sheaf. As explained to me by M.~Temkin, 
\begin{equation}  \label{e:goodnews}
\cF_\infty/p^n\cF_\infty=\cF_n
\end{equation}
(this will be proved in \S\ref{sss: proof of goodnews}). Moreover,
\begin{equation}  \label{e:2goodnews}
R\Gamma (\sY ,\cF_\infty)=R\Gamma  (\sY ,\cF),
\end{equation}
where $R\Gamma  (\sY ,\cF)$ was defined in \S\ref{sss:qcoh cohomology}: indeed, for any open affine $U\subset\sY$ the derived projective limit of $H^0(U,\sF_n)$ equals the usual one because the maps $H^0(U,\sF_{n+1})\to H^0(U,\sF_n)$ are surjective.

Because of \eqref{e:goodnews} and \eqref{e:2goodnews}, we will usually not distinguish $\cF_\infty$ from $\cF$.

Let $\cO_\sY$ denote the projective limit of the sheaves $\cO_{\sY_n}$. Then the ringed space $(\sY ,\cO_{\sY})$ is a formal scheme in the sense of EGA, and $\cF_\infty$ is a sheaf of $\cO_\sY$-modules.

\subsubsection{Proof of \eqref{e:goodnews}}  \label{sss: proof of goodnews}
It suffices to show that for every open affine $U\subset\sY$ the sequence
\[
H^0(U,\cF_\infty)\overset{p^n}\longrightarrow H^0(U,\cF_\infty)\to H^0(U,\cF_n)\to 0
\]
is exact. Let $M:=H^0(U,\cF_\infty)$ and  $M_n:=H^0(U,\cF_n)$. Then $M$ is the projective limit of~$M_n$. The transition maps $M_{n+1}\to M_n$ are surjective, so the map 
$M\to M_n$ is surjective for each~$n$.

Let  $M^n:=\Ker (M\epi M_n)$, so $M_n=M/M^n$. We have to show that $M^n=p^nM$.

Note that $M$ is separated and complete with respect to the filtration formed by $M^i$, $i\in\BN$.
Since $M_i=M_{i+1}/p^iM_{i+1}$, we have
\begin{equation}  \label{e: proof of goodnews}
M^i=p^iM+M^{i+1}.
\end{equation}
So $M^n=p^nM+M^{n+1}\supset p^nM$. On the other hand, if $x\in M^n$ then applying \eqref{e: proof of goodnews} for $i=n, n+1,n+2,\ldots\,$, we get
\[
x=p^nm_0+x_1, \quad x_1=p^{n+1}m_1+x_2,  \quad x_2=p^{n+2}m_2+x_3, \; \ldots
\]
for some elements $m_j\in M$ and $x_j\in M^{n+j}$. Then $x=p^nm$, where $m=\sum\limits_{j=0}^\infty p^jm_j$ (this infinite series converges because $p^jm_j\in M^j$ and $M$ is complete).  \qed

\subsection{Formulation of the results}
\subsubsection{Convention}
By a \emph{crystal} on an $\BF_p$-scheme $X$ we mean a crystal of quasi-coherent $\cO$-modules on the absolute crystalline site
$\Cris (X)$. For the definition of this site, see the end of \S III.1.1.3 of \cite{B74}. 

The proof of the following theorem will be given in \S\ref{ss:flatgroupoid}-\ref{ss:stacky cohomology}.

\begin{thm}   \label{t:crystals via stacks}
Let $X$ be a Frobenius-smooth scheme (i.e., an $\BF_p$-scheme satisfying the equivalent conditions of Lem\-ma~\ref{l:morally smooth}). 

(i) The simplicial formal scheme $\Acris_{\bullet}$ from \S\ref{sss:Acris} is the nerve\footnote{Recall that a groupoid  (and more generally, a category) is uniquely determined by its nerve.} of a flat affine\footnote{This means that the two canonical morphisms $\sG\to W(X_{\perf} )$ are flat and affine.}  groupoid $\sG$ acting on $\Acris_0=W(X_{\perf} )$.

(ii) A crystal on $X$ is the same as an object $M\in\QCoh (W(X_{\perf})/\sG )$ (or equivalently, a $\sG$-equivariant object of $\QCoh (W(X_{\perf}))$.

(iii) For $M$ as above, $R\Gamma_{\cris} (X,M)=R\Gamma (W(X_{\perf})/\sG ,M)$.
\end{thm}

In the case $M=\cO$ statement (iii) is equivalent to \cite[Thm.~1.10]{BMS}, but our proof is different (we think of crystalline cohomology in terms of the  \v {C}ech-Alexander complex\footnote{The  \v {C}ech-Alexander complex is discussed in  \cite[\S 5.1, 5.5]{Gr68},  \cite[\S 2]{BhdJ}, and \cite[\S 1.7-1.8]{B13}.}
rather than the de Rham complex).

\begin{rem}  \label{r:Tot}
An object $M\in\QCoh (W({X_{\perf})}/\sG )$ defines a collection of objects $M^n\in\QCoh (\Acris_n )$. For each $r\in\BN$, the sheaves $M^n/p^rM^n$ form a cosimplicial sheaf\footnote{More precisely, a cosimplicial sheaf of modules over the cosimplicial sheaf of rings formed by the structure sheaves of the formal schemes  $\Acris_n\,$.} on the topological space $X$.
Statement (iii) of the theorem can be reformulated as a canonical isomorphism
\begin{equation}   \label{e:Tot}
R\Gamma_{\cris} (X,M)\iso\underset{r}{\underset{\longleftarrow}\lim}  R\Gamma (X,\Tot (M^{\bullet}/p^rM^{\bullet})).
\end{equation}
\end{rem}

\begin{rem}   
The stacks $(W({X_{\perf})}/\sG)\otimes_{\BZ_p}(\BZ/p^r\BZ)$ are usually \emph{not} Artin stacks 
because the two canonical morphisms $\sG\to W({X_{\perf})}$ usually have infinite type.
\end{rem}

\subsection{Proof of Theorem~\ref{t:crystals via stacks}(i)}     \label{ss:flatgroupoid}
We can assume that $X=\Spec B$, where $B$ is an $\BF_p$-algebra with a finite $p$-basis. 

\subsubsection{The simplicial formal scheme $\sX_{\bullet}\,$}   \label{sss:sXbullet}
By \cite[Prop.~1.1.7]{BM} or by \cite[Remark~1.2.3(a)]{dJ}, $X$ admits a lift $\sX=\Spf\sB$, where $\sB$ is a flat $p$-adically complete $\BZ_p$-algebra with $\sB/p\sB=B$. 
Fix $\sX$. Let $\sX_n$ be the $p$-adically completed PD hull of the (ideal of the) diagonal
\[
X\mono X^{n+1}\subset \sX^{n+1}=\sX^{[n]}.
\]
Then $\sX_n$ is a $p$-adic formal scheme whose underlying topological space is $X$. The formal schemes $\sX_n$ form a simplicial formal scheme $\sX_{\bullet}\,$.

Fix a $p$-basis $x_1,\ldots ,x_d\in B$. 
Let $\tilde x_j\in\cB$, $\tilde x_j\mapsto x_j$. We have canonical embeddings
$i_m:\cB=H^0(\sX,\cO_{\sX})\mono H^0(\sX^{[n]},\cO_{\sX^{[n]}})$.

\begin{prop}   \label{p:flatness by BM}
As an $i_m(\cB )$-algebra, $H^0(\sX_n,\cO_{\sX_n})$ is the $p$-adically completed algebra of PD polynomials over $i_m(\cB )$ with respect to the elements
$i_r(\tilde x_j)-i_m(\tilde x_j)$, where $r\in\{0,\ldots ,n\}$, $r\ne m$, $1\le j\le d$.
\end{prop}

\begin{proof}
For $n=2$ this is \cite[Cor.~1.3.2(i)]{BM}. The general case follows. 
\end{proof}

\begin{cor}    \label{c:flatness by BM}
$\sX_{\bullet}$ is the nerve of a flat affine groupoid $\sR$ acting on $\sX$. \qed
\end{cor}

\begin{lem}   \label{l:PD hulls}
The formation of PD hulls commutes with flat base change.
\end{lem}

\begin{proof}
This  was proved by Berthelot \cite[Prop.~I.2.7.1]{B74}.
\end{proof}

\subsubsection{End of the proof}   \label{sss:tildepi}
By \cite[Prop.~1.2.6]{BM}, the morphism $\pi :X_{\perf}\to X$ lifts to a morphism $\tilde\pi :W(X_{\perf})\to\sX$. Since $\Fr_X:X\to X$ is flat, so are $\pi$ and  $\tilde\pi$.

Let $\Pow_n$ be as in \S\ref{sss:Pow}. The formal scheme $\Acris_n$ from  \S\ref{sss:Acris} is the $p$-adically completed PD hull of $\Pow_n$ in $W(X_{\perf})^{[n]}$. So by Lemma~\ref{l:PD hulls} and flatness of $\tilde\pi$, the diagram
\begin{equation}   \label{e:Cartesian by BM}
\CD
\Acris_n  @>\tilde\pi_n>> \sX_n   \\
@VVV    @VVV   \\
W(X_{\perf})^{[n]}   @>>> \sX^{[n]}
\endCD
\end{equation}
is Cartesian. So Corollary~\ref{c:flatness by BM} implies that $\Acris_{\bullet}$ is the nerve of a flat groupoid acting on~$W(X_{\perf})$. \qed

\begin{rem}    \label{r:iso of stacks}
In the situation of \S\ref{sss:tildepi}, $\tilde\pi :W(X_{\perf})\to\sX$ induces an isomorphism
\begin{equation}  \label{e:iso of stacks}
W(X_{\perf})/\sG\iso\sX/\sR ,
\end{equation}
where $\sG$ and $\sR$ are as in Theorem~\ref{t:crystals via stacks}(i) and Corollary~\ref{c:flatness by BM}, respectively. This follows from faithful flatness of $\tilde\pi$ and the fact that the diagram \eqref{e:Cartesian by BM} is Cartesian.
\end{rem}

\begin{rem}    \label{r:2iso of stacks}
The isomorphism \eqref{e:iso of stacks} does not depend on the choice of $\tilde\pi$. Indeed, if $\tilde\pi' :W(X_{\perf})\to\sX$ is another lift of $\pi$ then the morphism
$W(X_{\perf})\overset{(\tilde\pi,\tilde\pi' )}\longrightarrow\sX\times\sX$ factors through $\sX_1$; to see this, use the PD structure on the ideal of the subscheme 
$X_{\perf}\subset W(X_{\perf})$.
\end{rem}

\subsection{Proof of Theorem~\ref{t:crystals via stacks}(ii)}
\subsubsection{General remark}   \label{sss:def of Acris}
Recall that if $Y$ is a semiperfect $\BF_p$-scheme then $A_{\cris}(Y)\otimes (\BZ/p^r\BZ)$ is the final object of the crystalline site of $Y$ over $\BZ/p^r\BZ$ 
(see \cite[\S 2.2]{F94} or \cite[Prop.~2.2.1]{Dr}).
So a crystal on $Y$ is the same as an object $N\in\QCoh (A_{\cris}(Y))$, and the crystalline cohomology of the crystal is just $R\Gamma (A_{\cris}(Y),N)$.

\subsubsection{The functor in one direction}   \label{sss:in one direction}
Let us apply \S\ref{sss:def of Acris} to the semiperfect schemes $\Pow_n$ from \S\ref{sss:Pow}. Recall that $\Acris_n:=A_{\cris}(\Pow_n)$. So given a crystal on $X$, its pullback to $\Pow_n$ can be viewed as an object
$M^n\in\QCoh (\Acris_n)$. Moreover, the collection of objects $M^n$ is compatible via $*$-pullbacks\footnote{By this we mean that for every map $f:[m]\to [n]$ one has a canonical isomorphism $(f^+)^*M^m\iso M^n$, where $f^+:X^{[n]}\to X^{[m]}$ is induced by $f$, and these isomorphisms are compatible with composition of $f$'s.}.
By the definition of the groupoid $\sG$ (see Theorem~\ref{t:crystals via stacks}(i)), such a compatible collection is the same as a $\sG$-equivariant object of 
$\QCoh (\Acris_0)=\QCoh (W(X_{\perf}))$ or equivalently, an object of  $\QCoh(W(X_{\perf})/\sG)$. Thus we have constructed a functor
\begin{equation}  \label{e:in one direction}  
\{\mbox{Crystals on }X\}\to\QCoh (W(X_{\perf})/\sG).
\end{equation}

It remains to prove that the functor \eqref{e:in one direction} is an equivalence. The question is local, so we can assume that $X$ is the spectrum of an $\BF_p$-algebra with a finite $p$-basis.

\subsubsection{Factorizing the functor \eqref{e:in one direction}}   \label{sss:factorizing}
Let $\sX$ and $\sX_{\bullet}$ be as in \S\ref{sss:sXbullet}. 
Each scheme of the form $\sX_n\otimes(\BZ/p^r\BZ)$ is a PD thickening of $X$, so a crystal $M$ on $X$ defines for each $n$ an object $M_{\sX_n}\in\QCoh (\sX_n)$.
This collection of objects is compatible via $*$-pullbacks\footnote{Here we use that the map $\sX_n\otimes(\BZ/p^r\BZ)\to\sX_m\otimes(\BZ/p^r\BZ)$ corresponding to any map $[m]\to [n]$ is a PD morphism.}. By the definition of the groupoid $\sR$ (see Corollary~\ref{c:flatness by BM}), such a compatible collection is the same as an $\sR$-equivariant 
object of  $\QCoh (\sX )$ or equivalently, an object of $\QCoh (\sX/\sR )$.  Thus we have constructed a functor
\begin{equation}  \label{e:2in one direction}  
\{\mbox{Crystals on }X\}\to\QCoh (\sX/\sR).
\end{equation}
The functor \eqref{e:in one direction}  is the composition of \eqref{e:2in one direction} and the equivalence 
$$\QCoh (\sX/\sR)\iso\QCoh (W(X_{\perf})/\sG)$$ 
corresponding to \eqref{e:iso of stacks}. It remains to show that \eqref{e:2in one direction} is an equivalence. This is a consequence of the next lemma.

\begin{lem}    \label{l:"formal smoothness"}
Suppose that $X$ is the spectrum of a $\BF_p$-algebra with a finite $p$-basis. Let $\sX$ be as in \S\ref{sss:sXbullet}.
Then every object of the absolute crystalline site of $X$ admits a morphism to $\sX\otimes (\BZ/p^r\BZ)$ for some $r$.  
\end{lem}

\begin{proof}
Follows from \cite[Prop.~1.2.6]{BM}. 
\end{proof}

\subsection{Proof of Theorem~\ref{t:crystals via stacks}(iii)}    \label{ss:stacky cohomology}
Let $M$ be a crystal on $X$. We can also think of $M$ as an object of $\QCoh (W(X_{\perf})/\sG)$.

\subsubsection{General remark}
Let $X_{\cris}$ be the absolute crystalline topos of $X$ (i.e., the category of sheaves on $\Cris (X)$).
Let $Y_{\bullet}$ be a simplicial ringed topos over $X_{\cris}$. Then one has a canonical morphism
\begin{equation}  \label{e:2Tot}
R\Gamma_{\cris}(X,M)\to\Tot (R\Gamma (Y_{\bullet}, M^{\bullet})),
\end{equation}
where $M^n$ is the pullback 
of $M$ to $Y_n\,$. Moreover, this construction is functorial in $Y_{\bullet}\,$.

\subsubsection{The map in one direction}  \label{sss:map in one direction}
Apply \eqref{e:2Tot} for $Y_{\bullet}=(\Pow_{\bullet})_{\cris}\,$, where $\Pow_{\bullet}$ is as in \S\ref{sss:Pow}. 
As explained in \S\ref{sss:def of Acris}, we can think of $M^n$ as an object of   $\QCoh(\Acris_n)$, where $\Acris_n:=A_{\cris}(\Pow_n)$; it is this object that was denoted by $M^n$ in Remark~\ref{r:Tot}. Moreover, $R\Gamma_{\cris}(\Pow_n, M^n)=R\Gamma (\Acris_n ,M^n)$ by  \S\ref{sss:def of Acris}.
Thus we can rewrite \eqref{e:2Tot} as a morphism
\begin{equation}   \label{e:3Tot}
R\Gamma_{\cris}(X,M)\to\Tot (R\Gamma(\Acris_{\bullet}, M^{\bullet}))=R\Gamma (W(X_{\perf})/\sG, M).
\end{equation}
It remains to prove that the morphism \eqref{e:3Tot} is an isomorphism. The question is local, so we can assume that $X$ is the spectrum of an $\BF_p$-algebra with a finite $p$-basis. 

\subsubsection{End of the proof}
Let $\sX$ and $\sX_{\bullet}$ be as in \S\ref{sss:sXbullet}. By Lemma~\ref{l:"formal smoothness"},
we get an isomorphism
\begin{equation}    \label{e:4Tot}
R\Gamma_{\cris}(X,M)\iso\Tot (R\Gamma(\sX_{\bullet},  M_{\sX_{\bullet}}))=\Tot (\Gamma(\sX_{\bullet},  M_{\sX_{\bullet}}))
\end{equation}
where $M_{\sX_{\bullet}}$ is as in \S\ref{sss:factorizing}; $\Tot (\Gamma(\sX_{\bullet},  M_{\sX_{\bullet}}))$ is called the  \emph{\v {C}ech-Alexander complex}.

We have 
\[
\Tot (\Gamma(\sX_{\bullet},  M_{\sX_{\bullet}}))=R\Gamma (\sX/\sR, M),
\]
where $\sR$ is the groupoid from Corollary~\ref{c:flatness by BM} and $\sX/\sR$ is the quotient stack (just as in \S\ref{sss:factorizing}, we can think of $M$ as an object of
$\QCoh (\sX/\sR)$). It remains to show that \eqref{e:3Tot} is the usual isomorphism 
\[
R\Gamma (\sX/\sR, M)\to R\Gamma (W(X_{\perf})/\sG, M)
\]
corresponding to the isomorphism of stacks \eqref{e:iso of stacks}. This follows from the next lemma.

\begin{lem}  \label{l:general toposlogy}
Choose $\tilde\pi :W(X_{\perf})\to\sX$ as in \S\ref{sss:tildepi} and let $\tilde\pi_n :\Acris_n\to\sX_n$ be as in dia\-gram~\eqref{e:Cartesian by BM}. Then the map \eqref{e:3Tot} equals the composition of \eqref{e:4Tot} and the
morphism $$\Tot (R\Gamma(\sX_{\bullet},  M_{\sX_{\bullet}}))\to \Tot (R\Gamma(\Acris_{\bullet}, M^{\bullet}))$$ that comes from the maps 
$\tilde\pi_n^*:R\Gamma (\sX_n ,M_{\sX_n})\to R\Gamma (\Acris_n ,M^n)$.
\end{lem}

\begin{proof}
For each $n$, $\sX_n$ is an object of the crystalline topos $X_{\cris}$. Let $X_{\cris}/\sX_n$ be the category of objects of $X_{\cris}$ over $\sX_n$; this category is a ringed topos over $X_{\cris}$. As $n$ varies, we get a simplicial ringed topos $X_{\cris}/\sX_{\bullet}$ over $X_{\cris}$. Moreover, we have a morphism 
$$(\Pow_{\bullet})_{\cris}\to X_{\cris}/\sX_{\bullet}$$ 
of simplicial ringed topoi over $X_{\cris}\,$.

The morphism  \eqref{e:4Tot} is the morphism  \eqref{e:2Tot} for $Y_{\bullet}=X_{\cris}/\sX_{\bullet}$.
 So the lemma follows from functoriality of the map \eqref{e:2Tot} with respect to $Y_{\bullet}\,$.
\end{proof}

\subsection{$H^0_{\cris}(X,\cO)$ and the ring of constants}
As before, we assume that $X$ is Frobenius-smooth in the sense of \S\ref{ss:morally smooth}. Let us compute $H^0_{\cris}(X,\cO)$, where $\cO$ is the structure sheaf on the absolute crystalline site of $X$.

\subsubsection{The ring of constants}  \label{sss:2ring of constants}
Let $A$ be the ring of regular functions on $X$ and $k:=\bigcap\limits_{n=1}^\infty A^{p^n}$. We call $k$ the \emph{ring of constants} of $X$. 

Frobenius-smoothness implies that $X$ is reduced. So $A$ is reduced, and $k$ is a perfect $\BF_p$-algebra.

\begin{lem}  \label{l:hat G_a}
Let $X=\Spec B$, where $B$ is an $\BF_p$-algebra with a $p$-basis\footnote{The definition of $p$-basis was given in Lemma~\ref{ss:morally smooth}.} $x_1,\ldots ,x_m$. Let $\pi :X\to\BG_a^m$ be the morphism defined by $x_1,\ldots ,x_m$. Let $\hat\BG_a$ be the formal additive group over $\BF_p$. 

(i) There is a unique action of $\hat\BG_a^m$ on $X$ such that $\pi :X\to\BG_a^m$ is $\hat\BG_a^m$-equivariant.

(ii) The ring of constants on $X$ is equal to the ring of all $\hat\BG_a^m$-invariant regular functions on $X$.

(iii) A closed subscheme $Y\subset X$ is preserved by the $\hat\BG_a^m$-action if and only if
for each $n\in\BN$ there exists a closed subscheme $Z\subset X$ such that $Y=(\Fr^n_X)^{-1}(Z)$.
\end{lem}

\begin{proof}
(i) $\hat\BG_a$ is the inductive limit of the finite group schemes $\Ker (\BG_a^m\overset{\Fr^n}\longrightarrow\BG_a^m )$, $n\in\BN$. For every $n\in\BN$ the diagram
\[
\CD
X  @>\Fr^n>>X   \\
@V\pi VV    @VV\pi V   \\
\BG_a^m  @>\Fr^n>> \BG_a^m
\endCD
\]
is Cartesian (for $n=1$ by the definition of $p$-basis, the general case follows). Any action of $\Ker (\BG_a^m\overset{\Fr^n}\longrightarrow\BG_a^m )$ on $X$ is fiberwise with respect to the morphism $\Fr^n:X\to X$. So there is one and only one such action with the property that $\pi :X\to\BG_a^m$ is $\hat\BG_a^m$-equivariant.

(ii-iii) The morphism $\Fr^n:X\to X$ is a torsor with respect to $\Ker (\BG_a^m\overset{\Fr^n}\longrightarrow\BG_a^m )$.
\end{proof}

\begin{prop}  \label{p:constants}

Let $X$ be a Frobenius-smooth $\BF_p$-scheme. Let $k$ be the ring of constants of $X$. Then the  map $W(k)\to H^0_{\cris}(X,\cO)$ induced by the morphism $X\to\Spec k$ is an isomorphism. 
\end{prop}

\begin{proof}
We can assume that $X=\Spec B$, where $B$ is an $\BF_p$-algebra with a $p$-basis $x_1,\ldots ,x_m$. 
By \cite[Prop.~1.1.7]{BM}, there exists a flat $p$-adically complete $\BZ_p$-algebra $\tilde B$ with $\tilde B/p\tilde B=B$. For each $i$ choose a lift of $x_i$ to $\tilde B$; this lift will still be noted by $x_i\,$. By \cite[Prop.~1.3.1]{BM}, for each $n\in\BN$ the module of differentials of $\tilde B/p^n \tilde B$ is a free $(\tilde B/p^n \tilde B)$-module with basis $dx_i\,$, $1\le i\le m$. So for each $i\le m$ there is a unique derivation 
$D_i:\tilde B\to\tilde B$ such that $D_i(x_j)=\delta_{ij}\,$.  
One has
\[
H^0_{\cris}(X,\cO)=\bigcap_{i=1}^m\Ker (\tilde B\overset{D_i}\longrightarrow \tilde B);
\]
this follows, e.g., from \cite[Prop.~1.3.3]{BM} (because $H^0_{\cris}(X,\cO)$ is the $\BZ_p$-module of endomorphisms of the crystal $\cO$).

Since $x_1,\ldots ,x_m$ form a $p$-basis in $B$, for each $r\in\BN$ the ring $\tilde B$ is topologically generated by $x_1,\ldots ,x_m$ and elements of the form $\tilde f^{{p^r}}$, $\tilde f\in\tilde B$. This implies that for each $i\le m$ and $l\in\BN$ one has
\[
D_i^l(\tilde B)\subset l!\cdot\tilde B,
\]
so one has the commuting operators $D_i^{(l)}:=(l!)^{-1}D_i^l$ acting on $\tilde B$ and satisfying the Leibniz formula
\[
D_i^{(l)}(\tilde f\tilde g)=\sum_{a+b=l}D_i^{(a)}(\tilde f)D_i^{(b)}(\tilde g), \quad \tilde f,\tilde g\in\tilde B
\]
and the relation $D_i^{(r)}D_i^{(s)}=\binom{r+s}{r}D_i^{(r+s)}$. These operators define an action of $\hat\BG_a^m$ on $\tilde B$, where $\hat\BG_a$ is the additive formal group over $\BZ_p\,$. The corresponding action of $\hat\BG_a^m$ on $B$ is the one from Lemma~\ref{l:hat G_a}(i).

To prove that the map $W(k)\to H^0_{\cris}(X,\cO)$ is an isomorphism, it suffices to check that if $\tilde f\in\tilde B$ is killed by $D_1,\ldots,D_m$ and $f\in B$ is the image of  
$\tilde f$ then $f\in k$.  It is clear that $\tilde f$ is $\hat\BG_a^m$-invariant. So $f$ is $\hat\BG_a^m$-invariant. By Lemma~\ref{l:hat G_a}(ii), this means that $f\in k$.
\end{proof}

\section{Isocrystals}  \label{s:Isoc}
\subsection{A class of schemes}  \label{ss:noetherian morally smooth}
\begin{prop} \label{p:noetherian morally smooth}
The following properties of an  $\BF_p$-scheme $X$  are equivalent:

(i) $X$ is  Noetherian and Frobenius-smooth (see Lemma~\ref{l:morally smooth} and the sentence after it);

(ii) $X$ is  Noetherian, reduced, and $\Omega^1_X$ is locally free and finitely generated.

These properties imply that the scheme $X$ is regular and excellent.
\end{prop}

\begin{proof}
It is clear that (i)$\Rightarrow$(ii). 
The implication (ii)$\Rightarrow$(i) follows from  Theorem~1 of \cite{Ty}, whose proof is based on \cite[Prop.~1]{Fog}; the reducedness assumption in (ii) is necessary because the definition of ``having a $p$-basis'' used in \cite{Ty} is different from the one from Lemma~\ref{l:morally smooth}.
E.~Kunz proved that property (i) implies regularity and excellence, see Theorems 107-108 of 
\cite[\S 42]{Mat}.
\end{proof}

\emph{Throughout \,\S\ref{s:Isoc}, we assume that $X$ has the equivalent properties of Proposition~\ref{p:noetherian morally smooth}.}
For instance, if $k$ is a perfect field of characteristic $p$ one can take $X$ to be either $\Spec k[[x_1,\ldots ,x_n]]$ or a quasi-compact smooth $k$-scheme.

\subsubsection{The ring of constants}  \label{sss:ring of constants}
In \ref{sss:2ring of constants} we defined the ring of constants $k$ by $k:=\bigcap\limits_{n=1}^\infty A^{p^n}$, where $A$ is the ring of regular functions on $X$.
If $X$ has the properties of  Proposition~\ref{p:noetherian morally smooth} then $k=\bigcap\limits_{n=1}^\infty E^{p^n}$, where $E$ is the ring of rational functions on $X$: indeed, if
$f\in\bigcap\limits_{n=1}^\infty E^{p^n}$and $D$ is the divisor of poles of $f$ then $D$ is divisible by $p^n$ for all $n\in\BN$, so $D=0$. Thus if $X$ is irreducible then the ring of constants is a perfect field; in general, it is a product of perfect fields.

\subsection{Coherent crystals and isocrystals}  \label{ss:CohIsoc}
Let $X$ be an $\BF_p$-scheme that has the equivalent properties of Proposition~\ref{p:noetherian morally smooth}. 
By a \emph{coherent crystal} on an $\BF_p$-scheme $X$ we mean a crystal of finitely generated\footnote{Let us recall that according to the definition from EGA, ``finitely generated" really means ``locally finitely generated".} quasi-coherent $\cO$-modules on the absolute crystalline site of $X$. 

Let $\sG$ be the groupoid on $W(X_{\perf})$ constructed in Theorem~\ref{t:crystals via stacks}(i). We have the categories
\[
\QCoh_{\fin}^\flat (W(X_{\perf})/\sG )\subset\QCoh_{\fin} (W(X_{\perf})/\sG )\subset\QCoh (W(X_{\perf})/\sG ),
\]
(we are using the notation of \S\ref{sss:qcoh notation}-\ref{sss:qcoh-fin}).

\begin{lem}    \label{l:CohCrys}
(i) The equivalence \eqref{sss:in one direction} identifies the category of coherent crystals on $X$ with $\QCoh_{\fin}  (W(X_{\perf})/\sG )$.

(ii) The category $\QCoh_{\fin} (W(X_{\perf})/\sG )$ is abelian.

(iii) $\QCoh_{\fin}^\flat  (W(X_{\perf})/\sG )\otimes\BQ=\QCoh_{\fin} (W(X_{\perf})/\sG )\otimes\BQ$
\end{lem}

\begin{proof}
By Remark~\ref{r:iso of stacks}, $W(X_{\perf})/\sG$ is locally isomorphic to a quotient of a Noetherian formal scheme by a flat groupoid. This implies (ii-iii). Statement (i) follows from Lemma~\ref{l:"formal smoothness"}.
\end{proof}

\begin{defin}  \label{def:Isoc}
Objects of the category
\[
\Isoc (X):=\QCoh_{\fin} (W(X_{\perf})/\sG )\otimes\BQ=\QCoh_{\fin}^\flat  (W(X_{\perf})/\sG )\otimes\BQ
\]
will be called \emph{isocrystals} on $X$.
\end{defin}

(Thus our isocrystals are not necessarily convergent in the sense of \cite{O}.)

\subsection{Local projectivity}   
\begin{prop}  \label{l:local freeness}
Let $\sX$ be a $\BZ_p$-flat $p$-adic formal scheme with $\sX\otimes_{\BZ_p}\BF_p=X$. 
Let $M$ be a coherent crystal on $X$ and $M_{\sX}$ the corresponding coherent $\cO_{\sX}$-module. Then the $(\cO_{\sX}\otimes\BQ )$-module $M_{\sX}\otimes\BQ$ is locally projective. 
\end{prop}

A proof will be given in \S\ref{ss:proof of local freeness}.

\begin{cor}  \label{c:local freeness}
Let $M\in\QCoh_{\fin} (W(X_{\perf})/\sG )$. Let $M_{W(X_{\perf})}$ be the corresponding module over~$\cO_{W(X_{\perf})}$. Then $M_{W(X_{\perf})}\otimes\BQ$ is a locally projective\footnote{A module over a sheaf of rings is said to be locally projective if it can be locally represented as a direct summand of a free module.} module over $\cO_{W(X_{\perf})}\otimes\BQ$.
\end{cor}

\begin{proof}
The statement is local on $X$. So by \cite[Prop.~1.2.6]{BM}, we can assume that the canonical morphism $X_{\perf}\to X$ lifts to a morphism $W(X_{\perf})\to\sX$, where 
$\sX\in\FSch$, $\sX\otimes_{\BZ_p}\BF_p=X$. 
It remains to apply Proposition~\ref{l:local freeness}. 
\end{proof}

\begin{cor}  \label{c:Tannakian}
 If $X$ is irreducible then $\Isoc (X)$ is a Tannakian category over $\Frac W(k)$, where $k$ is the field of constants of $X$ (in the sense of \S\ref{sss:ring of constants}). 
\end{cor}

\begin{proof}
The fact that $\Isoc (X)$ is abelian  follows from Lemma~\ref{l:CohCrys}(ii).

$W(X_{\perf})/\sG$ is locally isomorphic to a quotient of a Noetherian formal scheme by a flat groupoid, so
the tensor category  $\QCoh_{\fin} (W(X_{\perf})/\sG )$ has internal  $\HHom$'s. Proposition~\ref{l:local freeness} implies that for any $M,N\in\QCoh_{\fin} (W(X_{\perf})/\sG )$, the canonical morphism
\[
\HHom (M,\cO )\otimes N\to \HHom (M,N )
\]
becomes an isomorphism in the category $\Isoc (X):=\QCoh_{\fin} (W(X_{\perf})/\sG )\otimes\BQ$. So $\Isoc (X)$ is rigid by \cite[Prop. 2.3]{De90}.

By Proposition~\ref{p:constants}, the endomorphism ring of the unit object of $\Isoc (X)$ equals $\Frac W(k)$.  

Let $\alpha:\Spec E\to X$ be a point with $E$ being a perfect field. Let $K:=\Frac W(E)$. Let
\[
\Phi_\alpha:\Isoc (X)\to\{K\mbox{-vector spaces\}}
\]
be the functor obtained by composing $\alpha^*:\Isoc (X)\to \Isoc (\Spec E)$ with the obvious fiber functor $\Isoc (\Spec E)\to$\{$K$-vector spaces\}. 
Then $\Phi_\alpha$ is an exact tensor functor. Since $X$ is connected, Proposition~\ref{l:local freeness} implies that if $M\in\Isoc (X)$ and $\Phi_\alpha (M)=0$ then $M=0$. So $\Phi_\alpha$ is a fiber functor.
\end{proof}

\subsection{Proof of Proposition~\ref{l:local freeness}}   \label{ss:proof of local freeness}
The proposition is well known if $X$ is a scheme of finite type over a perfect field, see \cite[Cor.~2.9]{O} or \cite[\S 2.3.4]{berthelot-rigid}. The proof given in \emph{loc. cit.} uses the following fact: if $E$ is a field of characteristic 0 and $N$ is a finitely generated module over $E[[t_1,\ldots, t_n]]$ which admits a connection, then $N$ is free.
The proof given below is essentially the same (but organized using Fitting ideals). 

\subsubsection{Strategy}
We can assume that $X=\Spec B$, $\sX =\Spf\tilde B$, where $B$ and $\tilde B$ are as in the proof of Proposition~\ref{p:constants}. Choose an exact sequence of $\tilde B$-modules
\[
0\to N\to\tilde B^l\to M_{\sX}\to 0.
\]
For $r\in\{0,\ldots ,l\}$, let $I_r\subset\tilde B$ be the image of the canonical map
\begin{equation}  \label{e:Fitting}
\bigwedge\nolimits^{l-r}N\otimes\Hom(\bigwedge\nolimits^{l-r}\tilde B^l, \tilde B)\to\tilde B
\end{equation}
(so the ideals $I_0\subset\ldots I_{l-1}\subset I_l=\tilde B$ are the Fitting ideals of $M_{\sX}$). Let 
\[
J_r:=\{u\in\tilde B\,|\, p^ju\in I_r \mbox{ for some }j\}.
\]
Define the ideals $J'_0\subset\ldots J'_{l-1}\subset J'_l=B$ by $J'_r:=J_r/pJ_r\,$. 
The $(\cO_{\sX}\otimes\BQ )$-module $M_{\sX}\otimes\BQ$ is locally projective of constant rank $r$ if and only if $J'_r=B$ and $J'_i=0$ for $i<r$.
So to prove Proposition~\ref{l:local freeness}, it suffices to show that the closed subscheme $\Spec (B/J'_r)\subset\Spec B$ is open for each $r$.

We will use the notation from the proof of Proposition~\ref{p:constants}; in particular, we have the derivations $D_i:\tilde B\to\tilde B$ and the differential operators $D_i^{(l)}:\tilde B\to\tilde B$ and $D_i^{(l)}:B\to B$. 

\begin{lem}  \label{l:preserved by D_i}
$D_i (I_r)\subset I_r$ for all $i$ and $r$.
\end{lem}

\begin{proof}
Our $M_{\sX}$ is a module over the ring of differential operators $B[D_1,\ldots , D_m]$. 
For each $i$ there exists a $\tilde B$-linear operator $L_i:\tilde B^l\to\tilde B^l$ such that the map $f:\tilde B^l\to M_{\sX}$ satisfies the identity
\[
D_i\circ f=f\circ\nabla_i \,, \mbox{ where } \nabla_i:=D_i+L_i\, .
\]
Then $\nabla_i (N)\subset N$. Think of $\nabla$ as a (not necessarily integrable) connection on $\tilde B^l$ and $N$. Equip $\tilde B$ with the trivial connection. Then \eqref{e:Fitting} becomes a horizontal morphism, so its image $I_r$ is preserved by the operators $D_i\,$.
\end{proof}

Our goal is to show that the closed subscheme $\Spec (B/J'_r)\subset\Spec B$ is open for each $r$. Since $B$ is Noetherian, it suffices to prove the following

\begin{lem}  
For each $r,n\in\BN$ there exists a closed subscheme $Z\subset\Spec B=X$ such that $\Spec (B/J'_r)=(\Fr^n_X)^{-1}(Z)$.
\end{lem}

\begin{proof}
Lemma~\ref{l:preserved by D_i} implies that $D_i^{(l)} (J_r)\subset J_r$ for all $i$ and $l$ (recall that  $D_i^{(l)}= D_i^l/l!\,$). So $D_i^{(l)} (J'_r)\subset J'_r\,$.
This means that the closed subscheme $\Spec B/J'_r\subset\Spec B$ is preserved by the $\hat\BG_a^m$-action from Lemma~\ref{l:hat G_a}(i).
It remains to use Lemma~\ref{l:hat G_a}(iii).
\end{proof}

\subsection{Isocrystals as vector bundles}   \label{ss:Bun_Q}
\subsubsection{The category $\Bun_{\BQ}(\sY )$}   \label{sss:Bun_Q}
Let $\FSch$ denote the category of $\BZ_p$-flat $p$-adic formal schemes. 
For $\sY\in\FSch$, let $\Bun_{\BQ}(\sY )$ denote the category of finitely generated locally projective $(\cO_{\sY}\otimes\BQ )$-modules.

We secretly think of $\Bun_{\BQ}(\sY )$ as the category of vector bundles on the ``generic fiber" $\sY\otimes\BQ$, which can hopefully be defined as some kind of ana\-lytic space (maybe in the sense of R.~Huber, see \cite{Hub} and \cite[\S 2.1]{SW}). This is true if $\sY$ is a formal scheme of finite type over $W(k)$, where $k$ is a perfect field of characteristic $p$ (moreover, in this case $\sY\otimes\BQ$ can be understood as an analytic space in the sense of Tate or Berkovich). But we need non-Noetherian formal schemes (e.g., the formal scheme $W(X_{\perf})$ from \S\ref{ss:CohIsoc}).

\begin{prop}  \label{p:global projectivity}
Suppose that $\sY\in\FSch$ is affine. Then any object of $\Bun_{\BQ}(\sY )$ is a direct summand of $\cO_{\sY}^n\otimes\BQ$ for some $n\in\BN$. 
\end{prop}

The proof is given in \S\ref{ss:descent for Bun}.

\subsubsection{Flat descent for $\Bun_{\BQ}(\sY )$}   \label{sss:descent for Bun_Q}
Let $C^\bullet$ be a cosimplicial category, i.e., a functor from the simplex category $\Delta$ to the 2-category of categories. Then the projective limit of this functor is denoted by 
$\Tot (C^\bullet )$. So an object of $\Tot (C^\bullet )$ is an object of $C^0$ with an isomorphism between its two images in $C^1$ satisfying the cocycle condition (whose formulation in\-volves~$C^2$).

In particular, for a simplicial object $\sY_\bullet$ in $\FSch$ we have the category $\Tot (\Bun_\BQ (\sY_\bullet ))$.

\begin{prop}     \label{p:descent for Bun_Q}
Let $f:\sY\to\sZ$ be a faithfully flat\footnote{A morphism of $p$-adic formal schemes $f:\sY\to\sZ$ is said to be flat if for every $r\in\BN$ it induces a flat morphism $\sY\otimes (\BZ/p^r\BZ)\to\sZ\otimes (\BZ/p^r\BZ)$ (if $\sY$ and $\sZ$ are $\BZ_p$-flat it suffices to check this for $n=1$). We do \emph{not} require $f$ to induce flat morphisms 
$H^0(\sZ',\cO_{\sZ'})\to H^0(\sY',\cO_{\sY'})$, where $\sY'\subset\sY$ and $\sZ'\subset\sZ$ are open affines such that $\sY'\subset f^{-1}(\sZ')$.} quasi-compact morphism in $\FSch$.
Let $\sY_n$ be the fiber product (over $\sZ$) of $n+1$ copies of $\sY$ (in particular, $\sY_0=\sY$). Then the functor
\begin{equation}  \label{e:descent for Bun_Q}
\Bun_\BQ (\sZ )\to\Tot (\Bun_\BQ (\sY_\bullet ))
\end{equation}
is an equivalence.
\end{prop}

\begin{proof}[Proof of full faithfulness]
If $\sY$ and $\sZ$ are affine then the sequence
\[
0\to H^0(\sZ ,\cO_{\sZ})\to H^0(\sY_0,\cO_{\sY_0})\rightrightarrows H^0(\sY_1,\cO_{\sY_1})
\]
is exact by usual flat descent. This implies that the functor \eqref{e:descent for Bun_Q} is fully faithful. Essential surjectivity will be proved in \S\ref{ss:descent for Bun}.
\end{proof}

Let us note that Proposition~\ref{p:descent for Bun_Q} was substantially generalized by A.~Mathew \cite{Akh}.

\subsubsection{Equivariant objects of $\Bun_{\BQ}(\sY )$}
Suppose we have a groupoid $\Gamma\rightrightarrows\sY$ in $\FSch$ such that the two morphisms $\Gamma\to\sY$ are flat and quasi-compact.
Let $\sY_\bullet$ be its nerve (this is a simplicial formal scheme with $\sY_0=\sY$ and $\sY_1=\Gamma$). Set
\[
\Bun_\BQ (\sY/\Gamma ):=\Tot (\Bun_\BQ (\sY_\bullet )).
\]
This notation is legitimate because by Proposition~\ref{p:descent for Bun_Q}, the category $\Tot (\Bun_\BQ (\sY_\bullet ))$ depends only on the quotient stack $\sY/\Gamma$: indeed, 
\[
\Tot (\Bun_\BQ (\sY_\bullet ))=\underset{\sZ}{\underset{\longleftarrow}\lim}\,  \Bun_\BQ (\sZ ),
\]
where $\sZ$ runs through the category of objects of $\FSch$ equipped with a morphism to~$\sY/\Gamma$.
Objects of $\Bun_\BQ (\sY/\Gamma )$ are called \emph{$\Gamma$-equivariant objects of $\Bun_{\BQ}(\sY )$.} 

Now let $\sY=W(X_{\perf})$, $\Gamma =\sG$. Let $\sY_\bullet$ be the nerve of $\sG$. By Corollary~\ref{c:local freeness}, the essential image of the functor
\begin{equation}   \label{e:isoc to Tot}
\Isoc (X)=\Tot(\QCoh_{\fin}^\flat  (\sY_\bullet ))\otimes\BQ\to\Tot(\QCoh_{\fin}^\flat  (\sY_\bullet )\otimes\BQ )
\end{equation}
is contained in $\Tot(\Bun_{\BQ}(\sY_\bullet ) )=\Bun_{\BQ}(W(X_{\perf})/\sG )$. So we get a tensor functor 
\begin{equation}   \label{e:isoc to BunQ}
\Isoc (X)\to\Bun_{\BQ}(W(X_{\perf})/\sG ). 
\end{equation}

\begin{thm} \label{t:isocrystals via stacks}
The functor \eqref{e:isoc to BunQ} is an equivalence.
\end{thm}

The proof will be given in \S\ref{sss:isocrystals via stacks}.

\subsection{Banachian games}   \label{ss:Banachian games}
We will use the notation and terminology of \S\ref{ss:qcoh notation}.
As explained in \S\ref{sss:bad&good news}, for a $p$-adic formal scheme $\sY$ it is harmless to identify an object $\cF\in\QCoh (\sY )$ with the corresponding sheaf $\sF_\infty$ on 
$\sY$, where $\cF_\infty$ is the projective limit of the sheaves $\cF_n=\cF/p^n\cF$. We will do it sometimes.

\subsubsection{One of the goals}
Let $\qFSch\subset\FSch$ be the full subcategory of quasi-compact quasi-separated formal schemes.
We will study the presheaves of categories
\[
\sY\mapsto \QCoh^\flat (\sY )\otimes\BQ \quad\mbox{ and } \quad \sY\mapsto \QCoh^\flat_{\fin} (\sY )\otimes\BQ , \quad \sY\in\qFSch\, .
\]
One of our goals is to prove that these presheaves are fpqc sheaves (see Corollary~\ref{c:descent for QCohQ} below).

\begin{rem}   \label{r:Banachian spaces}
$\QCoh^\flat (\Spf\BZ_p )\otimes\BQ$ is the category of \emph{Banachian spaces} over $\BQ_p\,$. By a Banachian space over a non-Archimedean field we mean a complete topological vector space such that the topology comes from a non-Archimedean norm. (Note that the topology determines  an equivalence class of norms but not a particular norm.)
\end{rem}

\begin{rem}   \label{r:Banachian modules}
If $\sY\in\FSch$ is affine then $\QCoh^\flat (\sY )\otimes\BQ$ is the category of Banachian modules over the Banach $\BQ_p$-algebra $H^0(\sY,\cO_{\sY})\otimes\BQ$. For any
$\sY\in\qFSch\,$, Corollary~\ref{c:descent for QCohQ} proved below allows one to identify $\QCoh^\flat (\sY )\otimes\BQ$ with the category of ``quasi-coherent sheaves of Banachian $(\cO_{\sY}\otimes\BQ)$-modules".
\end{rem}

\begin{prop}   \label{p:descent of finiteness}
Let $f:\sY'\to\sY$ be a faithfully flat morphism in  $\qFSch\,$. Let $M\in\QCoh^\flat (\sY )\otimes\BQ$. Then $M\in\QCoh^\flat_{\fin} (\sY )\otimes\BQ$ if and only if 
$f^*M\in\QCoh^\flat_{\fin} (\sY' )\otimes\BQ$.
\end{prop}

\begin{proof}
The ``only if" statement is obvious. 

Say that $M_1,M_2\in\QCoh^\flat (\sY )$ are \emph{isogenous} if they are isomorphic in 
$\QCoh^\flat (\sY )\otimes\BQ$. To prove the  ``if" statement, we have to show that if $M\in\QCoh^\flat (\sY )$ and $f^*M$ is isogenous to some $N'\in\QCoh^\flat_{\fin} (\sY' )$ then $M$ is isogenous to some $N\in\QCoh^\flat_{\fin} (\sY )$. We have 
$f^*M\supset N'\supset p^nf^*M$ for some $n$. The submodule $N'/p^{n+1}f^*M\subset f^*(M/p^{n+1}M)$ is finitely generated. Since $\sY\otimes (\BZ/p^{n+1}\BZ)$ is a quasi-compact quasi-separated scheme, $M/p^{n+1}M$ equals the sum of its finitely generated submodules. So there exists a finitely generated submodule $L\subset M/p^{n+1}M$ such that $f^*L\supset N'/p^{n+1}f^*M$. Since $N'\supset p^nf^*M$, we have $f^*L\supset f^*(p^nM/p^{n+1}M)$, so $L\supset p^nM/p^{n+1}M$ by fully faithfulness. Now let $N\subset M$ be the preimage of $L$. Then $N\supset p^nM$, and $N/p^{n+1}M=L$ is finitely generated. So $N$ is finitely generated. It is clear that $N$ is isogeneous to $M$.
\end{proof}

\begin{prop}   \label{p:equivariance up to isogeny}
Let $\Gamma\rightrightarrows\sY$ be a flat groupoid in $\qFSch$. Let $\sY_\bullet$ be its nerve. Then the functor
\begin{equation}  
\Tot ( \QCoh^\flat (\sY_\bullet ))\otimes\BQ\to\Tot ( \QCoh^\flat (\sY_\bullet )\otimes\BQ )
\end{equation}
is an equivalence.
\end{prop}

The proof given below is similar to that of \cite[Thm.~1.9]{O}.

\begin{proof}
Full faithfulness is clear. Let us prove essential surjectivity.

Let $\pi_0,\pi_1:\sY_1\to\sY_0$ be the face maps and $e:\sY_0\to\sY_1$ the degeneration map (i.e., the unit of the groupoid). An object of $\Tot ( \QCoh^\flat (\sY_\bullet )\otimes\BQ )$ is given by a pair $(M,\alpha )$,  where $M\in\QCoh^\flat (\sY_0)$ and $\alpha\in\Hom (\pi_0^*M,\pi_1^*M)\otimes\BQ$ satisfies the cocycle condition and the condition $e^*(\alpha)=\id_M$. Let $n\ge 0$ be such that  $p^n\alpha\in\Hom (\pi_0^*M,\pi_1^*M)$. We will construct an $\cO_{\sY}$-submodule $M'\subset M$ containing $p^nM$ such that 
$M'\in\QCoh (\sY )$ and the morphism $\alpha :\pi_0^*M\to p^{-n}\pi_1^*M$ maps $\pi_0^*M'$ to $\pi_1^*M'$. This is clearly enough.

Let us introduce some notation. For $i\le j$ we have the morphism $$\alpha:\pi_0^*(p^iM/p^jM)\to\pi_1^*(p^{i-n}M/p^{j-n}M).$$
It induces a morphism
\begin{equation}  \label{e:beta maps}
\beta :p^iM/p^jM\to\Phi (p^{i-n}M/p^{j-n}M), \quad \Phi:=(\pi_0)_*\pi_1^*.
\end{equation}

Now let $M'\subset M$ be the submodule containing $p^nM$ such that
\begin{equation}  \label{e:def of M'}
M'/p^nM=\Ker (M/p^nM\overset{\beta}\longrightarrow\Phi(p^{-n}M/M)).
\end{equation}
Then $\alpha (\pi_0^*M')\subset \pi_1^*M$, and $M'$ is the biggest submodule of $M$ with this property. Our goal is to prove that
\begin{equation}  \label{e:desired inclusion}
\alpha (\pi_0^*M')\subset \pi_1^*M'.
\end{equation}

Let us make some general remarks. Since $\sY_\bullet$ is the nerve of a flat groupoid, $\Phi$ is a left exact comonad acting on the category 
$\{ N\in\QCoh (\sY )\,|\, p^rN=0 \mbox{ for some }r\}$.
In particular, we have the coproduct $\mu :\Phi\to\Phi^2$. The morphisms \eqref{e:beta maps} satisfy the following identity\footnote{If $n=0$ this identity means that $p^iM/p^jM$ is a comodule over $\Phi$.}, which follows from the cocycle property of $\alpha$: the composite maps
\[ 
p^iM/p^jM\overset{\beta}\longrightarrow \Phi (p^{i-n}M/p^{j-n}M)\overset{\Phi (\beta )}\longrightarrow \Phi^2 (p^{i-2n}M/p^{j-2n}M) ,
\]
\[   
p^iM/p^jM\overset{\beta}\longrightarrow \Phi (p^{i-n}M/p^{j-n}M)\to\Phi (p^{i-2n}M/p^{j-2n}M)\overset{\mu}\longrightarrow\Phi^2 (p^{i-2n}M/p^{j-2n}M)
\]
are equal to each other. In particular, the composite maps
\begin{equation}   \label{e:2action comp1}
M/p^{2n}M\overset{\beta}\longrightarrow \Phi (p^{-n}M/p^nM)\overset{\Phi (\beta )}\longrightarrow \Phi^2 (p^{-2n}M/M) ,
\end{equation}
\begin{equation}     \label{e:2action comp2}
M/p^{2n}M\overset{\beta}\longrightarrow\Phi (p^{-n}M/p^nM)\to\Phi (p^{-2n}M/M) \overset{\mu}\longrightarrow \Phi^2 (p^{-2n}M/M) 
\end{equation}
are equal to each other. 

Let us now prove \eqref{e:desired inclusion}. By \eqref{e:def of M'}, $M'/p^{2n}M$ is equal to the preimage of $\Phi (M/p^nM)$ under the map
\begin{equation}   \label{e:beta_2n}
M/p^{2n}M\overset{\beta}\longrightarrow \Phi (p^{-n}M/p^nM).
\end{equation}
 So the first two arrows of \eqref{e:2action comp2} kill $M'/p^{2n}M$. Therefore the composite map 
\eqref{e:2action comp1} kills $M'/p^{2n}M$. So the morphism \eqref{e:beta_2n} maps $M'/p^{2n}M$ to
\[
\Ker (\Phi (M/p^nM)\overset{\Phi (\beta )}
\longrightarrow \Phi^2 (p^{-n}M/M))=\Phi (M'/p^nM).
\]
This means that the composite map $$\pi_0^*(M'/p^{2n}M)\overset{\alpha}\longrightarrow\pi_1^*(p^{-n}M/p^nM)\to\pi_1^*(p^{-n}M/M')$$ is zero, which 
is equivalent to \eqref{e:desired inclusion}.
\end{proof}

\begin{cor}     \label{c:descent for QCohQ}
Let $f:\sY\to\sZ$ be a faithfully flat morphism in $\qFSch\,$.
Let $\sY_n$ be the fiber product (over $\sZ$) of $n+1$ copies of $\sY$. Then the functors
\begin{equation}  \label{e:descent for QCohQ}
\QCoh^\flat  (\sZ )\otimes\BQ\to\Tot ( \QCoh^\flat (\sY_\bullet )\otimes\BQ ),
\end{equation}
\begin{equation}  \label{e:descent for finQCohQ}
\QCoh^\flat_{\fin}  (\sZ )\otimes\BQ\to\Tot ( \QCoh^\flat_{\fin} (\sY_\bullet )\otimes\BQ ),
\end{equation}
are equivalences.
\end{cor}

\begin{proof}
Full faithfulness is clear. To prove essential surjectivity of \eqref{e:descent for QCohQ}, combine usual flat descent with Proposition~\ref{p:equivariance up to isogeny} applied to the groupoid with nerve $\sY_\bullet\,$. Essential surjectivity of \eqref{e:descent for finQCohQ} follows from Proposition~\ref{p:descent of finiteness} and essential surjectivity of \eqref{e:descent for QCohQ}.
\end{proof}

\subsubsection{Proof of Theorem~\ref{t:isocrystals via stacks}} \label{sss:isocrystals via stacks}
By Corollary~\ref{c:descent for QCohQ}, the functor \eqref{e:isoc to Tot} is an equivalence. Theorem~\ref{t:isocrystals via stacks} follows. \qed

\begin{prop}   \label{p:Karoubian}
The categories $\QCoh^\flat  (\sY )\otimes\BQ$ and $\QCoh^\flat_{\fin}  (\sY )\otimes\BQ$ are Karoubian for every $\sY\in\qFSch\,$.
\end{prop}

\begin{proof}   
Let $M\in\QCoh^\flat (\sY )$, $r\ge 0$ an integer, $\pi\in p^{-r}\cdot\End M$, $\pi^2=\pi$. Set 
\[
M':=p^rM+(p^r\cdot\pi)(M)\subset M.
\]
Then $M'\in\QCoh^\flat (\sY )$, $p^rM\subset M'\subset M$. Moreover, $\pi (M')\subset M'$, so we have a direct sum decomposition
$$M'=\pi (M')\oplus (1-\pi)(M')$$
in $\QCoh^\flat (\sY )$. It gives the desired direct sum composition of $M$ in $\QCoh^\flat (\sY )\otimes\BQ$. 

If $M$ is in $\QCoh^\flat_{\fin}  (\sY )$ then so are $M'$ and $\pi (M')$.
\end{proof}

\subsection{Proof of Propositions \ref{p:global projectivity} and  \ref{p:descent for Bun_Q}}   \label{ss:descent for Bun}
\begin{lem}   \label{l:open-closed subschemes}
Let $\varphi :A\to B$ be a homomorphism of $\BZ_p$-flat $p$-adically complete rings. Assume that the homomorphism $A/p^rA\to B/p^rB$ induced by $\varphi$ is faithfully flat for all $r$ (or equivalently, for $r=1$). Let $I_A\subset A\otimes\BQ$ be an ideal, and let $I_B\subset B\otimes\BQ$ be the ideal generated by $I_A$. If $I_B$ is generated by an idempotent in $B\otimes\BQ$ then $I_A$ is generated by an idempotent in $A\otimes\BQ$.
 \end{lem}
 
 \begin{proof}
Let $e\in B\otimes\BQ$ be the idempotent that generates $I_B\,$. Then 
$$e\in\Ker (B\rightrightarrows B\hat\otimes_AB)\otimes\BQ=A\otimes\BQ.$$
Let us show that $I_A=e\cdot (A\otimes\BQ )$.

Since $A\subset B$ and $(1-e)\cdot I_B=0$, we have $(1-e)\cdot I_A=0$, so $I_A\subset e\cdot (A\otimes\BQ )$. It remains to show that
\begin{equation}   \label{e:unit_generation}
I_A+(1-e)\cdot (A\otimes\BQ)=A\otimes\BQ.
\end{equation}
Let $J:=A\cap (I_A+(1-e)\cdot (A\otimes\BQ))$. Then $J$ generates the unit ideal in $B\otimes\BQ$, so $p^r\in JB$ for some $r$. But the map
\[
A/(p^{r+1}A+J)\to B/(p^{r+1}B+J\cdot B)
\]
is injective by the faithful flatness assumption. So $p^r\in J+p^{r+1}A$. Therefore $p^r\in J$, which is equivalent to \eqref{e:unit_generation}.
 \end{proof}
 
 \begin{cor}   \label{c:open-closed subschemes}
Let $\varphi :A\to B$ be as in Lemma~\ref{l:open-closed subschemes}. Let $f:A^l\otimes\BQ\to A^m\otimes\BQ$ be an $A$-module homomorphism. Let 
$f_B:B^l\otimes\BQ\to B^m\otimes\BQ$ be the base change of $f$. If $\Coker f_B$ is a projective $(B\otimes\BQ )$-module then $\Coker f$ is a projective $(A\otimes\BQ )$-module. 
 \end{cor}
 
 \begin{proof}
For each $d$ apply Lemma~\ref{l:open-closed subschemes} to the ideal of $A\otimes\BQ$ generated by all minors of $f$ of order $d$.
 \end{proof}

\begin{lem}    \label{l:bundles as fg modules}
For every $\sY\in\qFSch\,$, the canonical functor
\begin{equation}   \label{e:obvious functor}
\QCoh^\flat_{\fin}(\sY)\otimes\BQ \to\{\mbox{all } (\cO_\sY\otimes\BQ)\mbox{-modules}\}
\end{equation}
is fully faithful, and its essential image contains $\Bun_\BQ (\sY )$.
\end{lem}

\begin{proof}   
Full faithfulness is clear. Let us show that every object $M\in\Bun_\BQ (\sY )$ belongs to the essential image of \eqref{e:obvious functor}. If $M$ is a direct summand of 
$(\cO_\sY\otimes\BQ)^n$ this is true by Proposition~\ref{p:Karoubian}. The general case follows by Corollary~\ref{c:descent for QCohQ}.
\end{proof}

Lemma~\ref{l:bundles as fg modules} provides a fully faithful embedding
$\Bun_\BQ (\sY )\mono\QCoh^\flat_{\fin}(\sY)\otimes\BQ$.
It commutes with pullbacks.

The following lemma implies Propositions \ref{p:global projectivity} and  \ref{p:descent for Bun_Q}.

\begin{lem}    \label{l:descent for Bun_Q}
Let $\sY,\sZ\in\FSch$ be affine and $f:\sY\to\sZ$ a faithfully flat morphism. Let $M\in\QCoh^\flat_{\fin}(\sZ)\otimes\BQ$. Suppose that $f^*M\in\Bun_\BQ (\sY )$. Then $M$ is a direct summand of $\cO_\sZ^n\otimes\BQ$ for some $n$. 
\end{lem}

\begin{proof}   
Let $M_0\in\QCoh^\flat_{\fin}(\sZ)$ be a representative for $M$. Since $\sZ$ is affine, $M_0$ is generated by finitely many global sections, so we get an exact sequence
\[
0\to N_0\to\cO_{\sZ}^m\to M_0\to 0
\]
in $\QCoh^\flat(\sZ)$. Let $N\in\QCoh^\flat(\sZ)\otimes\BQ$ correspond to $N_0$. Then 
$$f^*N\in\Bun_{\BQ}(\sY)\subset\QCoh^\flat_{\fin}(\sY)\otimes\BQ,$$ 
so 
$N\in\QCoh^\flat_{\fin}(\sZ)$ by Proposition~\ref{p:descent of finiteness}. Thus we have an exact sequence
\[
(\cO_{\sZ}\otimes\BQ)^l\overset{\varphi}\longrightarrow (\cO_{\sZ}\otimes\BQ)^m\to M\to 0
\]
inducing an exact sequence $(\cO_{\sY}\otimes\BQ)^l\to (\cO_{\sY}\otimes\BQ)^m\to f^*M\to 0$. It remains to apply Corollary~\ref{c:open-closed subschemes}.
\end{proof}

\appendix

\section{The isomorphism between $W(X_{\perf})/\sG$ and the prismatization of~$X$}     \label{appendix}
\subsection{The goal}
\emph{Prismatization} is a certain functor $X\mapsto\WCart_X$ from the category of $p$-adic formal schemes to the category of stacks, see \cite{BL} or  
\cite[\S 1]{Prismatization}. Let us note that in \cite{Prismatization} this functor is denoted by $X\mapsto X^\prism$.

The stack $\WCart_X$ can always be presented (in many different ways) as a quotient of a formal scheme by a flat groupoid. The goal of this Appendix is to deduce from 
\cite{BL} that if $X$ is a Frobenius-smooth $\BF_p$-scheme then $\WCart_X$ has a \emph{canonical} presentation of this type; namely, one has a canonical isomorphism 
\begin{equation}    \label{e:WCart_X=W(X_perf)/sG}
\WCart_X\iso W(X_{\perf})/\sG,
\end{equation}
 as promised in Remark~\ref{r:prismatization}.

\subsection{Prismatization of semiperfect $\BF_p$-schemes}
\subsubsection{Perfect case}   \label{sss:Perfect case} 
According to \cite[Example~3.12]{BL}, for any perfect $\BF_p$-scheme $X$ one has a canonical isomorphism $W(X)\iso\WCart _X$, where
$W(X)$ is the $p$-adic formal scheme whose underlying topological space is that of $X$ and whose structure sheaf is obtained by applying to 
$\cO_X$ the functor of $p$-typical Witt vectors. 

\subsubsection{General case}   \label{sss:General case} 
More generally, for any semiperfect $\BF_p$-scheme $X$ one has a canonical isomorphism $A_{\cris}(X)\iso\WCart_X$, where $A_{\cris}(X)$ is the $p$-adic formal scheme whose underlying topological space is that of $X$ and whose structure sheaf is obtained by applying Fontaine's functor $A_{\cris}$ to $\cO_X$. 
This is \cite[Lemma~6.1]{BL} in the case that $X$ is regular semiperfect (the only one that we need) and  \cite[Cor.~7.18]{BL} for arbitrary semiperfect schemes. 

In these statements from \cite{BL} $A_{\cris}$ is not mentioned explicitly. Instead, it is proved that if $X$ is affine then $\WCart_X$ is the $\Spf$ of the prismatic cohomology of $X$. But it is known from \cite[\S 5]{BS} that the prismatic cohomology of an $\BF_p$-scheme identifies with its crystalline cohomology; on the other hand, if $R$ is a semiperfect $\BF_p$-algebra then the crystalline cohomology of $\Spec R$ equals $A_{\cris} (R)$ (in particular, it lives in degree zero\footnote{In general, the ``true'' objects are the \emph{derived} crystalline cohomology of $\Spec R$ and the \emph{derived} version of $A_{\cris} (R)$. They are canonically isomorphic and live in nonpositive degrees; if $R$ is \emph{regular} semiperfect then they live in degree zero and are equal to their classical prototypes.}).

\subsection{The morphism $W(X_{\perf})\to \WCart_X$}
By \ref{sss:Perfect case}, 
For any $\BF_p$-scheme $X$, the morphism 
\begin{equation}  \label{e:X_perf to X}
X_{\perf}\to X
\end{equation}
induces a morphism
\begin{equation}  \label{e:W(X_perf to WCart (X)}
W(X_{\perf})=\WCart_{X_{\perf}}\to\WCart_X
\end{equation}
(we have used \S\ref{sss:Perfect case}). 

From now on, suppose that $X$ is Frobenius-smooth in the sense of \S\ref{ss:morally smooth}. Then \eqref{e:X_perf to X} is a quasi-syntomic cover, so by \cite[Lemma~6.3]{BL},
the map \eqref{e:W(X_perf to WCart (X)} is  an fpqc cover (by this we mean that for every scheme $S$, every morphism $S\to \WCart_X$ lifts to a morphism 
$S\to W(X_{\perf})$ fpqc-locally on $S$).

\subsection{The  \v {C}ech nerve of \eqref{e:W(X_perf to WCart (X)}}
By definition, $\sG$ is the groupoid acting on $W(X_{\perf})$ whose \v {C}ech nerve is the simplical formal scheme $\Acris_{\bullet}\,$ from \S\ref{sss:Acris}.
So to finish constructing the isomorphism \eqref{e:WCart_X=W(X_perf)/sG}, it remains to identify the \v {C}ech nerve of $\eqref{e:W(X_perf to WCart (X)}$ with $\Acris_{\bullet}\,$.

Since $X$ is Frobenius-smooth, the morphism \eqref{e:X_perf to X} is flat. So by  \cite[Rem.~3.5]{BL}, the \v {C}ech nerve of $\eqref{e:W(X_perf to WCart (X)}$ is
$\WCart _{\Pow_{\bullet}}\,$, where $\Pow_{\bullet}$ is the \v {C}ech nerve of \eqref{e:X_perf to X}. By \S\ref{sss:General case}, we have
$\WCart_{\Pow_{\bullet}}=A_{\cris}(\Pow_{\bullet})$. Finally,  $A_{\cris}(\Pow_{\bullet})=:\Acris_{\bullet}\,$.

\bibliographystyle{alpha}

\begin{thebibliography}{BFM}

\bibitem[B13]{B13}
A.~Beilinson,  \textit{On the  crystalline period map}, Camb. J. Math. {\bf 1}, no. 1, 1-51 (2013).


\bibitem[B74]{B74}
P.~Berthelot, Cohomologie cristalline des sch\'emas de caract\'eristique $p>0$. 
Lecture Notes in Mathematics, {\bf 407}, Springer-Verlag, Berlin-New York, 1974.


\bibitem[B96]{berthelot-rigid} P. Berthelot, Cohomologie rigide et cohomologie rigide \`a support propre, part 1, Pr\'epublication IRMAR 96-03, available at \url{https://perso.univ-rennes1.fr/pierre.berthelot/}.

\bibitem[BBM]{BBM} P. Berthelot, L.~Breen, and W.~Messing,Th\'eorie de Dieudonn\'e cristalline. II.  Lecture Notes in Mathematics, {\bf 930}, Springer-Verlag, Berlin, 1982.

\bibitem[BM]{BM} P. Berthelot and W.~Messing, Th\'eorie de Dieudonn\'e cristalline. III. Th\'eor\`emes d'\'equivalence et de pleine fid\'elit\'e. 
In: The Grothendieck Festschrift, Vol. I, pp. 173--247, Progress in Math.{\bf 86}, Birkh\"auser, Boston, 1990. 

\bibitem[BhdJ]{BhdJ} B.~Bhatt and A.~J.~de Jong, \textit{Crystalline cohomology and de Rham cohomology}, arXiv:1110.5001.

 \bibitem[BL]{BL} B.~Bhatt and  J.~Lurie, \emph{The prismatization of $p$-adic formal schemes}, arXiv:2201.06124, , version 1.

 
 \bibitem[BS]{BS} B.~Bhatt and  P.~Scholze, 
\emph{Prisms and Prismatic Cohomology}, arXiv:1905.08229. 




\bibitem[BMS]{BMS} B.~Bhatt, M.~Morrow, and P.~Scholze, \emph{Topological Hochschild homology and integral p-adic Hodge theory},  Publ. Math. IHES
{\bf 129} (2019), 199--310.


\bibitem[dJ]{dJ} A.~J.~de Jong, \textit{Crystalline Dieudonn\'e module theory via formal and rigid geometry}, Publ. Math. IHES  {\bf 82} (1995), 5--96.


\bibitem[De90]{De90} P.~Deligne, \textit{Cat\'egories tannakiennes}, In: Grothendieck Festschrift, Vol. II, Progr. Math., {\bf 87}, 111-195, Birkh\"auser, Boston, 1990. 


\bibitem[Dr1]{Dr} V.~Drinfeld,  {\em On a theorem of Scholze-Weinstein}, arXiv:1810.04292.

   \bibitem[Dr2]{Prismatization} V.~Drinfeld, \emph{Prismatization},  arXiv:2005.04746.



\bibitem[Fog]{Fog} J.~Fogarty, {\em K\"ahler differentials and Hilbert's fourteenth problem for finite groups}, Amer. J. Math. {\bf 102} (1980), no. 6, 1159-1175. 

\bibitem[F94]{F94} J.-M.~Fontaine, \textit{Le corps des p\'eriodes p-adiques}. In: P\'eriodes p-adiques, 59--111, Ast\'erisque {\bf 223},  Soc. Math. France, Paris, 1994.


\bibitem[Gr68]{Gr68}  A.~Grothendieck,  \textit{Crystals and the de Rham cohomology of schemes},  In: Dix expos\'es sur la cohomologie des sch\'emas, 306--358, 
Adv. Stud. Pure Math. {\bf 3}, North-Holland, Amsterdam, 1968. 


\bibitem[Hub]{Hub}  R.~Huber,   \textit{A generalization of formal schemes and rigid analytic varieties}, Math. Zeitschrift  {\bf 217} (1994), no. 4, 513--551. 

\bibitem[Mat]{Akh}
A.~Mathew, \textit{Faithfully flat descent of almost perfect complexes in rigid geometry},  J. Pure Appl. Algebra {\bf 226} (2022), no. 5.


\bibitem[M]{Mat} H.~Matsumura,  \textit{Commutative algebra}, Mathematics Lecture Note Series {\bf 56}, The Benjamin/Cummings Publishing Company, Reading, Massachusetts, 1980.

\bibitem[O]{O}
A.~Ogus, \textit{$F$-isocrystals and de Rham cohomology II---convergent isocrystals},
Duke Math. J. {\bf 51} (1984), 765--850.


\bibitem[SW]{SW} P.~Scholze and J.~Weinstein, \textit{Moduli of p-divisible groups}, Camb. J. Math. {\bf 1} (2013), no. 2, 145--237. 
\bibitem[Ty]{Ty} A.~Tyc, {\em Differential basis, $p$-basis, and smoothness in characteristic $p>0$},
Proc. AMS~{\bf 103} (1988), no. 2, 389--394.

\end{thebibliography}

\end{document}